\documentclass[12pt]{article}
\usepackage{amssymb,amsmath,amsthm}
\usepackage{graphicx}
\usepackage{subcaption}
\usepackage{fullpage}
\usepackage{scrextend}
\usepackage{siunitx}
\usepackage{enumerate}
\usepackage{tikz,calc}
\usepackage{tikz-cd}
\usetikzlibrary{automata,positioning}
\usepackage{subcaption}
\usetikzlibrary{calc}
\usepackage{epsfig,graphicx}
\usepackage{listings}
\usepackage{enumerate}
\usepackage{float}
\usepackage{diagbox}
\usepackage{url}
\usepackage{makecell}
\usepackage{float}

 \newtheorem{theorem}{Theorem}%[section]
\newtheorem{prop}[theorem]{Proposition}
\newtheorem{lemma}[theorem]{Lemma}
\newtheorem{cor}[theorem]{Corollary}

\newcommand{\ind}{\mathbf{1}}

\newcommand{\supp}{\text{supp}}

\title{An almost-tight $L^2$ autoconvolution inequality}
\author{Ethan Patrick White \\ \small{The University of British Columbia}}
\date{}

\begin{document}

\maketitle

\begin{abstract}

Let $\mathcal{F}$ denote the set of functions $f \colon [-1/2,1/2] \to \mathbb{R}$ such that $\int f = 1$. We determine the value of $\inf_{f \in \mathcal{F}} \| f \ast f \|_2$ up to a 0.0014\% error, thereby making progress on a problem asked by Ben Green. Furthermore, we prove that a unique minimizer exists. As a corollary, we obtain improvements on the maximum size of $B_h[g]$ sets for $(g,h) \in \{ (2,2),(3,2),(4,2),(1,3),(1,4)\}$. 

\end{abstract}

\vspace{0.5cm}

%Let $g,N \in \mathbb{N}$ and $A \subset [N]$. We say $A$ is \emph{$g$-Sidon} if for all $m \in \mathbb{Z}$ 
%\[ \# \{ (a,b) \in A^4 \colon a+b = m\} \leq g.\]

\section{Introduction}

Let $g,h,N$ be positive integers. A subset $A \subset [N] = \{1,2,\ldots,N\}$ is an \emph{$B_h[g]$-set} if for every $x \in \mathbb{Z}$, there are at most $g$ representations of the form 
\[ a_1 + a_2 + \ldots + a_h = x, \quad a_i \in A, \quad 1 \leq i \leq h, \]
where two representations are considered to be the same if they are permutations of each other. As a shorthand, we let $B_h = B_h[1]$. Note that $B_2$ sets are the very well known \emph{Sidon sets}. Let $R_h[g](N)$ denote the largest size of subset $A \subset [N]$ such that $A$ is a $B_h[g]$ set. By counting the number of ordered $h$-tuples of elements of $A$, we have the simple bound $\binom{|A|+h-1}{h} \leq hN$, which implies $R_h[g](N) \leq (ghh!N)^{1/h}$. On the other hand, Bose and Chowla showed that there exist $B_h$ sets of size $N^{1/h}(1+o(1))$, where we use $o(1)$ to denote a term going to zero as $N \to \infty$. This lower bound result has been generalized to more pairs $(g,h)$ by several authors, \cite{CGT,CJ,LIND}. Recently the bound $R_h[g](N) \geq (Ng)^{1/h}(1-o(1))$ for all $N$, $g \geq 1$, and $h\geq 2$ was obtained in \cite{GTT}. In general, estimating the constant
\[ \sigma_h(g) = \lim_{N \to \infty} \frac{R_h[g](N)}{(gN)^{1/h}},\]
is an open problem. In fact, the only case for which the above limit is known to exist is in the case of the classical Sidon sets, where we have $\sigma_2(1) = 1$. Henceforth we will understand upper and lower bounds on $\sigma_h(g)$ to be estimates on the $\limsup$ and $\liminf$, respectively. \\

Several improved upper bounds for $\sigma_h(g)$ have been obtained by various authors, for references to many of them, as well as an excellent resource on $B_h[g]$ sets in general, see~\cite{KO}. In this work we will improve the upper bounds on $\sigma_h(g)$ for $h = 2$ and $2 \leq g \leq 4$, as well as $g = 1$ and $h = 3,4$. No improvement on $\sigma_h(g)$ for $g = 1$ and $h=3,4$ has been made since~\cite{BG} in 2001. The most recent improvements on estimates for $\sigma_2(g)$ which are the best for $2 \leq g \leq 4$ are the following.

% try out a table describing progress?

% \leq \sqrt{1.864(2g-1)}   & \text{Cilleruelo, Ruzsa \& Trujillo \cite{CRT}}

\begin{align*}
\sigma_2(g) & \leq \sqrt{1.75(2g-1)}   & \text{Green \cite{BG}}\\
  & \leq \sqrt{1.74217(2g-1)}  & \text{Yu \cite{YU}}\\
 & \leq \sqrt{1.740463(2g-1)}  & \text{Habsieger \& Plagne \cite{HP}}
\end{align*}

On the other hand, for $g \geq 5$, the best recent upper bounds are below.

\begin{align*}
\sigma_2(g) & \leq \sqrt{3.1696g}   & \text{Martin \& O'Bryant \cite{MO}}\\
  & \leq \sqrt{3.1377g}  & \text{Matolcsi \& Vinuesa \cite{MV}}\\
 & \leq \sqrt{3.125g}  & \text{Cloninger \& Steinerberger \cite{CS}}
\end{align*}

Interestingly, the key to improving upper bounds on $\sigma_2(g)$ is to better estimate the 2-norm of an autoconvolution for small $g$ and infinity norm of an autoconvolution for large $g$. In the case of the infinity norm, the best known results are 
\[0.64 \leq \inf_{\substack{f \colon [-1/2,1/2] \to \mathbb{R}_{\geq 0} \\ \int f \ = \ 1}} \| f \ast f\|_\infty \leq 0.75496.\]
The lower bound is due to Cloninger \& Steinerberger \cite{CS} and the upper bound is due to Matolcsi \& Vinuesa~\cite{MV}. It is believed that the upper bound above is closer to the truth. The method of Cloninger \& Steinerberger is computational, and is limited by a nonconvex optimization program.

Throughout we will use $\mathcal{F}$ to denote the family of functions $f \in L^1(-1/2,1/2)$ such that $\int f = 1$. For $1 \leq p \leq \infty$ we define 
\[ \mu_p = \inf_{f \in \mathcal{F}} \| f \ast f \|_p =\inf_{f \in \mathcal{F}}\left( \int_{-1}^1 \left| \int_{-1/2}^{1/2} f(t)f(x-t) dt \right|^p  dx\right)^{1/p} ,\]
the minimum $p$-norm of the autoconvolution supported on a unit interval. Determining $\mu_p$ precisely is an open problem for all values of $p$, and is the content of Problem 35 in Ben Green's \emph{100 open problems} \cite{BGP}. Prior to this work, the best known bounds for the $p = 2$ case are 
\[ 0.574575 \leq \mu_2^2 \leq 0.640733,\]
where the lower bound is due to Martin \& O'Bryant \cite{MO2} and the upper bound is due to Green~\cite{BG}. Our main theorem is the following; we improve upper and lower bounds for $\mu_2$. 

\begin{theorem}\label{MT} 
The $L^2$ norm of the autoconvolution $f \ast f$ for $f \in \mathcal{F}$ can be bounded as follows.
\[  \emph{\textbf{0.5746}}35728 \leq \mu_2^2  \leq \emph{\textbf{0.5746}}43711.\]

\end{theorem}

We are also able to prove that there exists a unique minimizer $f \in \mathcal{F}$ of the $L^2$ norm of an autoconvolution. This combined with methods of Green~\cite[Theorems 15, 17, 24]{BG} give the following corollary.

\begin{cor}\label{NTcor} The following asymptotic bounds on $B_h[g]$ sets hold.

\[ \sigma_2(g) \leq \left(\frac{2-1/g}{\mu_2^2}\right)^{1/2} ,\quad \sigma_3(1) \leq \left(\frac{2}{\mu_2^2} \right)^{1/3}, \quad\sigma_4(1) \leq \left(\frac{4}{\mu_2^2} \right)^{1/4}.\]

\end{cor} 

Due to Theorem~\ref{MT}, we have improved upper bounds for $\sigma_h[g]$ for $h=2$ and $2 \leq g \leq 4$ as well as $g = 1$ and $h = 3,4$. One of the main theorems proved by Green in~\cite{BG} gives bounds on the additive energy of a discrete function on $[N]$. Our bound in the continuous case yields the following.

\begin{cor}\label{AEcor} Let $H \colon [N] \to \mathbb{R}$ be a function with $\sum_{j=1}^N H(j) = N$. Then for all sufficiently large $N$ we have 
\[ \sum_{a+b = c+d} H(a)H(b)H(c)H(d) \geq \mu_2^2 N^3.\]

\end{cor}

The methods of Habsieger \& Plagne and Yu \cite{HP,YU}, Cloninger \& Steinerberger \cite{CS}, and Martin \& O'Bryant are all limited by long computation times. In contrast, the key to our improvement is a convex optimization program whose optimum value is shown to converge to $\mu_2^2$. The strategy of using Fourier analysis to produce a convex program to obtain bounds on a convolution-type inequality was also employed recently by the author to improve bounds on Erd\H{o}s' minimum overlap problem in \cite{EPW}. We hope that our methods may also be useful in estimations of autoconvolution for other $p$-norms.\\

% Outline of thhis paper... 

\section{Qualitative results}

In this section we prove the existence and uniqueness to the solution of the following optimization problem. 
\begin{align}
\textsc{Infimum: } &  \int_{-1}^{1} (f \ast f)^2(x) \ dx\label{objfunc} \\ 
\textsc{Such that: }&  f \in L^1, \quad \supp f \subset [-1/2,1/2], \quad  \int_{-1/2}^{1/2} f(x) \ dx = 1 .\label{qualcon}
\end{align} 

For all $f \in L^1(\mathbb{R})$ we define the Fourier transform on $\mathbb{R}$ as 
\[ \tilde{f}(y) = \int e^{-2\pi i xy}f(x) \ dx.\]

For any $f$ as in \eqref{qualcon}, we note that $\widetilde{f \ast f} = \tilde{f}^2$ and so by Parseval's identity
\begin{equation}\label{OFF} \int_{-1}^{1} (f \ast f)^2(x) \ dx = \int |\tilde{f}(y)|^4 \ dy.\end{equation}

We remark that \eqref{qualcon} defines the family $\mathcal{F}$ seen in the introduction. The following proposition proves the existence and uniqueness of an optimizer in $\mathcal{F}$ to \eqref{objfunc} using the `direct method in the calculus of variations'.

\begin{prop}\label{exun} 

There exists a unique extremizing function $f \in \mathcal{F}$ to the optimization problem \eqref{objfunc}. \end{prop}

\begin{proof} Let $\{f_n\} \subset \mathcal{F}$ be a minimizing sequence such that ${\lim_{n \to \infty } \| f_n \ast f_n \|_2 = \mu_2}$. Since $L^1$ and $L^\infty$ are separable we can apply the sequential Banach-Alaoglu Theorem to conclude the existence of $f \in L^{1}(-1/2,1/2)$ and $g \in L^\infty(\mathbb{R})$ such that 
\begin{align*}
f_n &\stackrel{\ast}{\rightharpoonup} f & \text{converges weakly in } L^{1}(-1/2,1/2), \\
\tilde{f_n} &\stackrel{\ast}{\rightharpoonup} g & \text{converges weakly in } L^{\infty}(\mathbb{R}),
\end{align*}
where possibly we passed to a subsequence of $\{f_n\}$ to make the above hold. For all $h \in L^{1}(\mathbb{R})$, by definition of convergence in the weak topology we have
\[ \langle g,h \rangle = \lim_{n \to \infty} \langle \tilde{f}_n ,h \rangle =  \lim_{n \to \infty} \langle {f}_n ,\tilde{h} \rangle =  \langle {f} , \tilde{h} \rangle= \langle \tilde{f} , h \rangle, \]
hence $g = \tilde{f}$. Note that for all $y \in \mathbb{R}$ we have
\[ \tilde{f_n}(y) = \int f_n(x)e^{-2\pi i x y} \ind_{(-1/2,1/2)}(x) \ dx, \]
and since $e^{-2\pi i x y} \ind_{(-1/2,1/2)}(x) \in L^\infty(\mathbb{R})$, by weak convergence we see $\lim_{n \to  \infty} \tilde{f_n}(y) = \tilde{f}(y)$. In addition, $\lim_{n \to  \infty} |\tilde{f_n}(y)|^4 = |\tilde{f}(y)|^4$ and so by Fatou's lemma:
\[ \int  |\tilde{f}(y)|^4 \ dy \leq \liminf_{n \to \infty} \int  |\tilde{f_n}(y)|^4\ \ dy .\]
Finally, we have 
\[1 = \lim_{n \to \infty} \langle \ind_{(-1/2,1/2)}, f_n \rangle = \langle \ind_{(-1/2,1/2)},f \rangle = \int_{-1/2}^{1/2} f(x) \ dx.\] 
We conclude that $f \in \mathcal{F}$ is an extremizing function. For uniqueness, suppose that $f,g \in \mathcal{F}$ satisfy $\|\tilde{f}\|_4 = \| \tilde{g}\|_4 = \mu_2^{1/2}$. Then by Minkowski's inequality,
\[ \left\| \frac{\tilde{f} + \tilde{g}}{2} \right\| \leq \frac{1}{2} \left( \|\tilde{f} \| + \|\tilde{g}\| \right) = \mu_2^{1/2}.\]
Minkowski's inequality above must be an equality, implying $f$ and $g$ are linearly dependant. Since $f,g$ have the same average value, we conclude that $f = g$ and so the extremizing function is unique.

\end{proof} 

Note that the uniqueness of the optimizer implies that it must be even. Throughout we will denote the unique optimizer by $f^\Diamond \in \mathcal{F}$. Next we prove that there is an $L^2$ function with objective value \eqref{objfunc} arbitrarily close to $\mu_2^2$. 

% $L^2$ functions are close to the optimizer.

%%%%%%%%%%%%%%%%%%%%%%%%%%%%%%%%%
%%%%%%%%%%%%%%%%%%%%%%%%%%%%%%%

\begin{lemma}\label{L2lem} We can add the constraint $f \in L^2$ to \eqref{qualcon} without changing the optimum. 

\end{lemma}

\begin{proof} Define the following family of rectangular functions for all $\epsilon >0$.
\[ h_\epsilon (x) = \begin{cases} 1/\epsilon & \text{if } -\frac{\epsilon}{2} < x < \frac{\epsilon}{2} \\ 1/(2\epsilon) & \text{if } x = \pm \frac{\epsilon}{2} \\ 0 & \text{otherwise.} \end{cases}\]
Note that the Fourier transform of $h_\epsilon$ is the sinc function
\[ \tilde{h}_\epsilon(y)  = \frac{\sin(\pi \epsilon y)}{\pi \epsilon y}.\]
Define
\[ f_\epsilon(x) = (1 + \epsilon)f^\Diamond \ast h_\epsilon((1+\epsilon)x), \quad \text{and consequently} \quad \tilde{f}_\epsilon(y) = \tilde{f}^\Diamond \tilde{h}_\epsilon\left( \frac{y}{1+\epsilon} \right).\]
It is easy to check that $f_\epsilon$ satisfies \eqref{qualcon}. Moreover, since $\| \tilde{f} \|_\infty \leq \| f^\Diamond \|_1$ we have by Parseval's theorem
\[ \| f_\epsilon\|_2^2 = \|\tilde{f}_\epsilon\|_2^2 = \int \left|  \tilde{f}^\Diamond \tilde{h}_\epsilon\left( \frac{y}{1+\epsilon} \right) \right|^2 \leq (1+\epsilon)\|\tilde{f}^\Diamond\|_\infty^2 \|h_\epsilon\|_2^2 = \frac{1+\epsilon}{\epsilon}\| f^\Diamond \|_1 ,\]
and so $f_\epsilon \in L^2$. Since $\| \tilde{h}_\epsilon\|_\infty \leq 1$ we have
\[ \int |\tilde{f}_\epsilon(y)|^4 \ dy \leq \int \left|\tilde{f}^\Diamond\left( \frac{y}{1+\epsilon}\right)\right|^4 \ dy = (1+\epsilon) \int |\tilde{f}^\Diamond(y)|^4 \ dy \leq (1+\epsilon)\mu_2^2.\]
This gives the result as $\epsilon$ can be made arbitrarily small.

\end{proof} 

Our final qualitative result shows that the three-fold convolution of the optimal function is constant on a large interval. The analogous discrete version of this result was proved in \cite{BG} using a similar argument.

\begin{lemma}\label{constantconv} Let $f^\Diamond\in \mathcal{F}$ denote the extremizing function for \eqref{qualcon}, and $F^\Diamond$ denote the extension of $f^\Diamond$ to a period 2 function, defined on $[-1,1]$ as $F^\Diamond(x) = \ind_{(-1/2,1/2)}(x)f^\Diamond(x)$. Then $F^\Diamond \ast F^\Diamond \ast F^\Diamond$ is constant on $(-1/2,1/2)$ with value $\mu_2^2/4$.

\end{lemma}

\begin{proof} By Parseval's identity we have
\begin{equation}\label{percon} \int_{-1}^1 (F^\Diamond \ast F^\Diamond)^2(x) \ dx = 2\sum_{k \in \mathbb{Z}} |\hat{F^\Diamond}|^4 = \frac{1}{4} \| f^\Diamond \ast f^\Diamond\|_2^2 = \mu_2^2/4.\end{equation}
Let $G$ be any period 2 function such that $G(x) = 0$ for all $x \in [-1,1]\setminus [-1/2,1/2]$ and $\int_{-1}^1 G(x) \ dx = 0$. By optimality of $F^\Diamond$ we have

\begin{align*} \lim_{\epsilon \to 0} \ & \frac{1}{\epsilon} \Bigg( \int^1_{-1} ((F^\Diamond +\epsilon G )\ast (F^\Diamond + \epsilon G))^2(x) \ dx  -  \int^1_{-1} (F^\Diamond \ast F^\Diamond)^2(x) \ dx \Bigg) \\
& = 4 \int_{-1}^1 (F^\Diamond \ast F^\Diamond)(x) (F^\Diamond \ast G)(x) \ dx = 4\int_{-1/2}^{1/2} F^\Diamond \ast F^\Diamond \ast F^\Diamond(x) G(x) \ dx = 0. \\
\end{align*}

The second last equality above used that $F^\Diamond$ is even. Clearly the final equality can only hold if $F^\Diamond \ast F^\Diamond \ast F^\Diamond$ is constant on $(-1/2,1/2)$. We may choose $G = F^\Diamond -\ind_{(-1/2,1/2)}$, and apply the above to yield
\[ \int_{-1/2}^{1/2} F^\Diamond \ast F^\Diamond \ast F^\Diamond(x) \left(F^\Diamond -\ind_{[-1/2,1/2]}\right)(x) \ dx =0. \]
By \eqref{percon} we conclude
\[ \mu_2^2 = 4\int_{-1}^1 (F^\Diamond \ast F^\Diamond)^2(x) \ dx = 4\int_{-1/2}^{1/2}  F^\Diamond \ast F^\Diamond \ast F^\Diamond(x) \ dx .\]

\end{proof}

\section{Useful identities}

For ease of notation we will always use lowercase letters $f,g$ to denote functions on $[-1/2,1/2]$, or period 1 functions. We define the Fourier transform of $f \colon [-1/2,1/2] \to \mathbb{R}$ for $k \in \mathbb{Z}$ as 
\[ \hat{f}(k) = \int_{-1/2}^{1/2} e^{-2\pi i k x} f(x) \ dx.\]

We will use upper case letters $F,G$ to denote functions on $[-1,1]$ or period 2 functions. We define the Fourier transform of $F \colon [-1,1] \to \mathbb{R}$ for $k \in \mathbb{Z}$ as 
\[ \hat{F}(k) = \frac{1}{2}\int_{-1}^{1} e^{-\pi i k x} f(x) \ dx.\]

This is an abuse of the notation `\ $\hat{}$\ ', but which of the two above transforms is meant will be made clear by the letter case of the function notation. Let $f \in \mathcal{F}$ and define $F(x)$ be the extension of $f(x)$ to a function on $[-1,1]$ defined by setting $F(x)= 0$ outside of $[-1/2,1/2]$. Since $\supp(F) \subset [-1/2,1/2]$, the support of $F \ast F$ is contained in $[-1,1]$, hence 
\[ \widehat{F\ast F}(k) = \frac{1}{2} \int_{-1}^1 e^{-\pi i kx}F\ast F(x) \ dx =  \frac{1}{2} \int_{-1}^1 e^{-\pi i kx}\int_{-1/2}^{1/2} f(t)f(x-t) \ dt \ dx \]
\begin{equation*}\label{M2F} =\frac{1}{2}  \int_{-1/2}^{1/2} e^{-\pi i k t} F(t) \int_{-1}^1 e^{-\pi i k(x-t)}F(x-t) \ dx \ dt = 2 \hat{F}(k)^2.\end{equation*}

%Define Fourier transforms for $k \in \mathbb{Z}$ normalized as follows:

%\[ \hat{F}(k) = \frac{1}{2} \int_{-1}^1 e^{-\pi i kx}F(x) \ dx,  \quad \hat{M}(k) = \frac{1}{2} \int_{-1}^1 e^{-\pi i kx}M(x) \ dx.\]

We calculate the relationship between $\hat{F}$ and $\hat{f}$ below.
\begin{equation}\label{f2Ffour} \hat{F}(m) = \frac{1}{2} \sum_{k}\hat{f}(k)\int_{-1/2}^{1/2} e^{\pi i x(2k-m)} \ dx = \begin{cases} \frac{1}{2} \hat{f}(m/2) & \text{if } m \text{ is even} \\ (-1)^{(m+1)/2}\sum_{k\in \mathbb{Z}} \frac{\hat{f}(k)(-1)^k}{\pi  (2k-m)} & \text{if } m \text{ is odd.} \end{cases}\end{equation}
From Parseval's theorem and the above we obtain
\begin{equation*}\label{Meq0}\|F\ast F\|_2^2 =2 \sum_{k \in \mathbb{Z}} |\widehat{F\ast F}(k)|^2 = 8\sum_{k \in \mathbb{Z}} |\hat{F}(k)|^4.\end{equation*}
\begin{equation}\label{Meq1} =\frac{1}{2} \sum_{m \in \mathbb{Z}} |\hat{f}(m)|^4 +\frac{8}{\pi^4} \sum_{\substack{m \in \mathbb{Z} \\ m \text{ odd}}} \left| \sum_{k \in \mathbb{Z}} \frac{\hat{f}(k)(-1)^k}{2k-m} \right|^4.\end{equation}

Since $f(x)$ is real and even, we know that for all $k \in \mathbb{Z}$, $\hat{f}(k) = \hat{f}(-k) \in \mathbb{R}$.

\begin{lemma}\label{FFid} For all $f \in \mathcal{F}$ we have the identity 

\begin{equation*} \|f \ast f \|_2^2 =\frac{1}{2} + \sum_{m=1}^\infty \hat{f}(m)^4 + \frac{16}{\pi^4}\sum_{\substack{m \geq 1 \\ m \text{ odd}}} \left( \frac{1}{m} +2 \sum_{k=1}^\infty \frac{m\hat{f}(k)(-1)^k}{m^2-4k^2} \right)^4.  \end{equation*}

\end{lemma} 

\begin{proof} Since $\hat{f}(0) = 1$, we have for all odd $m \in \mathbb{Z}$:
\[ \sum_{k\in \mathbb{Z}} \frac{\hat{f}(k)(-1)^k}{m-2k} = \frac{1}{m} + 2\sum_{k=1}^\infty \frac{m\hat{f}(k)(-1)^k}{m^2-4k^2}. \]
Substituting the above into \eqref{Meq1} gives the result. 

\end{proof}

We are unable to analytically determine the $\hat{f}(k)$ such that $\| f \ast f\|_2$ is minimized. In the following section we will use Lemma~\ref{FFid} together with a convex program to provide upper bounds on $\mu_2$ as well as an assignment of $\hat{f}(k)$ that is very close to optimal. The following lemma gives a method of obtaining strong lower bounds, from good $f \in \mathcal{F}$ with small $\|f \ast f\|_2$, i.e. good lower bounds can be found from good upper bound constructions.

\begin{lemma}\label{holdlem} Let $f,g$ be periodic real functions with period 1, such that $\int_{-1/2}^{1/2} f  = 1$ and $\int_{-1/2}^{1/2} g  = 2$. Define $F,G \colon [-1,1] \to \mathbb{R}$ by 
\[ F(x) = \begin{cases} f(x) & \text{for } x \in [-1/2,1/2] \\ 0 & \text{otherwise} \end{cases},\qquad G(x) = \begin{cases} 1 & \text{for } x \in [-1/2,1/2] \\ 1-g(x) & \text{otherwise.} \end{cases}\]
Then 
\begin{equation}\label{LBholder} 1/2 = \sum_{k \neq 0} \hat{F}(k) \overline{\hat{G}(k)} \leq \left(\sum_{k \neq 0} |\hat{F}(k) |^4\right)^{1/4}\left(\sum_{k \neq 0} |\hat{G}(k) |^{4/3}\right)^{3/4},\end{equation}
and the inequality is equality if and only if $f = f^\Diamond$ and $\hat{g}(k) = \frac{2}{1-2\mu_2^2 }\hat{f^\Diamond}(k)^3$ for all $k \neq 0$  where $f^\Diamond \in \mathcal{F}$ is the extremizer. 

\end{lemma}

\begin{proof} By Plancherel's theorem we have 
\[ 1 = \int_{-1}^1 F(x)\overline{G(x)} \ dx = \langle F,G\rangle = 2 \langle \hat{F}, \hat{G} \rangle = 2\sum_{k \in \mathbb{Z}} \hat{F}(k)\overline{\hat{G}(k)}.\]
Since $\hat{G}(0) = 0$, by applying H\"older's inequality we conclude \eqref{LBholder}. Let $F^\Diamond$ be viewed as a period 2 function on $\mathbb{R}$. Then by Lemma~\ref{constantconv} we see that 
\[ G^\Diamond(x) = \frac{8}{2\mu_2^2-1} F^\Diamond \ast F^\Diamond \ast F^\Diamond(x) - \frac{1}{2\mu_2^2-1},\]
is a function of the required form. Moreover, for all $k \neq 0$ we have $\hat{G^\Diamond} (k) = \frac{8}{2\mu_2^2-1} \hat{F^\Diamond}^3(k)$, and so H\"older's inequality will be tight. We can determine the function $g$ that gives rise to $G^\Diamond$. For all $k \neq 0$ we have
\[ \hat{G}^\Diamond(2k) = \frac{1}{2} \int_{-1}^1 e^{-2\pi i kx} G^\Diamond(x) \ dx = -\frac{1}{2} \int_{1/2}^{3/2} e^{-2\pi i kx} g(x) \ dx = -\frac{1}{2} \hat{g}(k) .\]
Consequently the $g(x)$ corresponding to $G^\Diamond$ have the Fourier coefficients 
\begin{equation}\label{f2g} \hat{g}(k) = -2 \hat{G}^\Diamond(2k) = \frac{-16}{2\mu_2^2-1} \hat{F^\Diamond}^3(2k) = \frac{2}{1-2\mu_2^2} \hat{f}^\Diamond(k)^3.\end{equation}

\end{proof}

The use of the Lemma~\ref{holdlem} is that if we obtain a $f$ close to optimal, then not only do we have a good upper bound estimate of $\mu_2^2$, but also a strong lower bound estimate.

\section{Quantitative results}\label{progsec}

In this section we describe a convex program used to approximate the optimal solution of \eqref{objfunc} with finitely many variables. Our primal program is the following.

%Let $T$ and $R$ be positive integers. %The variables of our program are $\{c_k,w_k,x_k\}_{k=1}^T, \{y_m,z_m,\epsilon_m\}_{m=1}^R$. 

\begin{align}
\textsc{Input: } & R,T \in \mathbb{N} \nonumber  \\
\textsc{Variables: }&  \{f_k,w_k,x_k\}_{k=1}^T, \{y_m,z_m\}_{m=1}^R\nonumber \\
\textsc{Minimize: } & \frac{1}{2} + \sum_{m=1}^T x_k + \frac{16}{\pi^4}\sum_{m=1}^R z_m \nonumber \\
\textsc{Subject to: }  &w_k \geq f_k^2, x_k \geq w_k^2; \quad 1 \leq k \leq T,\nonumber \\
&y_m \geq \left(\frac{1}{2m-1} + 2\sum_{k=1}^T \frac{(2m-1)f_k(-1)^k}{(2m-1)^2-4k^2}\right)^2  ; \quad 1 \leq m \leq R,\nonumber \\
&z_m \geq y_m^2  ; \quad 1 \leq m \leq R.\nonumber \\
%&|\epsilon_m|  \leq \frac{2m}{T^{5/4}}; \quad 1 \leq m \leq R.\label{epsiloncon} 
\end{align}

For any $R,T \in \mathbb{N}$, let $\mathcal{O}(R,T)$ be the optimum of the above program. We remark that the reason for the `redundant' variables $\{w_k,x_k\}_{k=1}^T$ and $\{y_m,z_m\}_{m=1}^R$ is to demonstrate that the program is easily implemented as a quadratically constrained linear program. For any $T \in \mathbb{N}$, let $\mathcal{F}_T \subset \mathcal{F}$ be functions that are degree at most $T$ in their Fourier series expansion, i.e. $f \in \mathcal{F}_T$ implies $\hat{f}(k) = 0$ for $|k|>T$. 

%simply so that the program has a linear objective and convex quadratic constraints. %This form of program is easily converted to a SOCP. 

\begin{prop}\label{progconv} For all $R,T \in \mathbb{N}$ the following chain of inequalities holds. 

\[ \min_{f \in \mathcal{F}_T} \| f \ast f \|_2^2 - 5T^{-1/3} \leq \mu_2^2 \leq \min_{f \in \mathcal{F}_T} \| f \ast f \|_2^2 \leq \mathcal{O}(T,R) + 5(T/R)^3.\]

Moreover, $|\mathcal{O}(R,\sqrt{R}) - \mu_2^2| < 10R^{-1/6}.$

\end{prop} 

\begin{proof} Fix arbitrary $R \geq 2T \in \mathbb{N}$. The middle inequality is trivial. We prove the first inequality first. Let $f^\Diamond \in \mathcal{F}$ be the extremizer, and define for all $0<\epsilon<1/4$,
\[ f_\epsilon(x) = (1 + \epsilon)f^\Diamond \ast h_\epsilon((1+\epsilon)x) \in \mathcal{F},\]
with $h_\epsilon$ as in Lemma~\ref{L2lem}. Note that since $\epsilon<1/4$ we can consider $h_\epsilon$ as a function on $(-1/2,1/2)$ with mass 1. Let $f_{\epsilon,T} \in \mathcal{F}_T$ be the degree $T$ Fourier approximation of $f_\epsilon$, i.e. 
\[ f_{\epsilon,T} (x)= \sum_{|k|\leq T} \hat{f}_\epsilon(k) e^{2\pi i k x} .\]
By Fourier inversion we have the estimate for all $y \in \mathbb{R}$:
\begin{align} |\tilde{f}^\Diamond(y)| & = \int_{-1/2}^{1/2} e^{-2 \pi i x y} \int_{-\infty}^{\infty} \tilde{f}^\Diamond(z)e^{2\pi i xz} \ dz \ dx\nonumber \\
& = \int_{-\infty}^\infty \tilde{f}^\Diamond(z) \int_{-1/2}^{1/2} e^{2\pi i x(z-y)} \ dx \ dz \nonumber\\
& = \int_{-\infty}^\infty \tilde{f}^\Diamond(z) \frac{\sin(\pi(z-y))}{\pi(z-y)} \ dz \nonumber\\
& \leq \|\tilde{f}^\Diamond\|_4 \cdot \|\sin(\pi z)/(\pi z)\|_{4/3} \leq 3,\label{ftsup} 
\end{align}
where in final inequality we used that $\|\tilde{f}^\Diamond\|_4 < 2/3$. Hence $\|\tilde{f}^\Diamond\|_\infty \leq 3$. We bound the $L^2(-1/2,1/2)$ norm of $f_{\epsilon,T}-f_\epsilon$ by Parseval's theorem.
\begin{align} \| f_{\epsilon,T} - f_{\epsilon} \|_2^2 &= \sum_{|k|>T} |\hat{f}_\epsilon(k)|^2= \sum_{|k|>T} |\tilde{f}_\epsilon(k)|^2 \nonumber\\
&= \sum_{|k|>T} |\tilde{f}^\Diamond(k/(1+\epsilon))\tilde{h}_\epsilon(k/(1+\epsilon))|^2 \nonumber\\
& \leq \|\tilde{f}^\Diamond\|_\infty \sum_{|k|>T} \left(\frac{1+\epsilon}{\pi \epsilon k}\right)^2 \leq \frac{1}{\epsilon^2 T}.\label{feptdif}
\end{align}
Combining \eqref{ftsup} and \eqref{feptdif} gives the estimate 
\[ \| \tilde{f}_{\epsilon,T} - \tilde{f}_\epsilon \|_4^4 \leq 36\| \tilde{f}_{\epsilon,T} - \tilde{f}_\epsilon \|_2^2 \leq \frac{36}{\epsilon^2 T}. \]
Consequently 
\[ \| \tilde{f}_{\epsilon,T} \|_4^4 \leq \| \tilde{f}_\epsilon \|_4^4 + \frac{36}{\epsilon^2 T} \leq (1+\epsilon)\mu_2^2 + \frac{36}{\epsilon^2 T}. \]
By choosing $\epsilon = (36/(\mu_2^2 T))^{1/3}$ to minimize the error, we obtain $\| \tilde{f}_{\epsilon,T} \|_4^4 \leq \mu_2^2 + 5/T^{1/3}$. We now prove the third inequality. Let $\{f_k\}_{k=1}^T$ be the solution to the program with inputs $R,T$. Define $f_P(x) = \sum_{|k| \leq T} e^{2 \pi i k x} f_{|k|}$. We have
 \begin{equation}\label{Ttoprog} \min_{f \in \mathcal{F}_T} \| f \ast f \|_2^2 \leq \|f_P \ast f_P\|_2^2 \leq \mathcal{O}(R,T) + \frac{16}{\pi^4}\sum_{m=R+1}^\infty \left( \frac{1}{2m-1} +2 \sum_{k=1}^T \frac{(2m-1)f_k(-1)^k}{(2m-1)^2-4k^2} \right)^4.\end{equation}
 For all $m \geq R+1$ and $1 \leq k \leq T$ we have $(2m-1)^2-4k^2 \geq 2m^2$. We can bound the inside sum above by H\"older's inequality. 
 \[ \left|\sum_{k=1}^T \frac{(2m-1)f_k(-1)^k}{(2m-1)^2-4k^2}\right| \leq (2m-1) \left( \sum_{k=1}^Tf_k^4 \right)^{1/4} \left( \sum_{k=1}^T(2m^2)^{-4/3} \right)^{3/4} \leq T^{3/4}/m. \]
 Substituting this estimate into \eqref{Ttoprog} we obtain
 \[ \min_{f \in \mathcal{F}_T} \| f \ast f \|_2^2 \leq \mathcal{O}(R,T) + \frac{16}{\pi^4}\sum_{m=R+1}^\infty \left(3T^{3/4}/m \right)^4 \leq \mathcal{O}(R,T)+5(T/R)^3. \]
By setting $T = \sqrt{R}$ we see that $\mathcal{O}(R,T) - 5R^{-1/6} \leq \mu_2^2 \leq \mathcal{O}(R,T) + 5R^{-1/6}$. This concludes the proof.

 \end{proof}

 As a consequence of Proposition~\ref{progconv} we see that the optimum of our program will converge to $\mu_2^2$, thereby giving both upper and lower bounds for $\mu_2^2$. The rate of convergence we have proved is not particularly fast, and so to obtain strong lower bounds, we will use the solution of our program in tandem with Lemma~\ref{holdlem}. \\

%\noindent {\bf \large Computational results} 

\subsection{Computational results}

In Proposition~\ref{progconv} we proved the convergence of the optimum $\mathcal{O}(R,\sqrt{R})$ to $\mu_2^2$. The convergence is not fast enough to generate good numerical results. To find good choices for the Fourier coefficients $\hat{f}(k)$, we run the convex program when $T/R$ is large. Our results can be improved slightly with more computation time, but our best current function comes from using our convex program with $R = 4000$ and $T = 30000$. We used IBM's CPLEX software on a personal computer to determine the optimal solution, and the full assignment of $\{f_k\}_{k=1}^{30000}$ is available upon request. The first 20 values of $f_k$ are displayed below. 

\begin{table}[H]\centering
\small{
 \begin{tabular}{| c |c | } 
 \hline
          $c_k$ : $1 \leq k \leq 10$ & $c_k$ : $11 \leq k \leq 20$ \\
 \hline
-0.297647135 & -0.094224143\\
\hline
0.216257252 & 0.090249054\\
\hline
-0.178150116 & -0.086738237\\
\hline
0.154960786 & 0.083607805\\
\hline
-0.138963721 & -0.080793794\\
\hline
0.12707629 & 0.07824618\\
\hline
-0.117795585 & -0.075925462\\
\hline
0.11029022 & 0.073799849\\
\hline
-0.104058086 & -0.071843457\\
\hline
0.09877573 & 0.0700349\\
\hline
\end{tabular}}
\caption{First values of $\{f_k\}$ for almost optimal $f(x)$.}
\label{coscoef}
\end{table}

In Figure~\ref{plot} we display the degree 20 trigonometric polynomial described in Table~\ref{coscoef}.

\begin{figure}[H]
 \centering
    \includegraphics[width=0.7\textwidth]{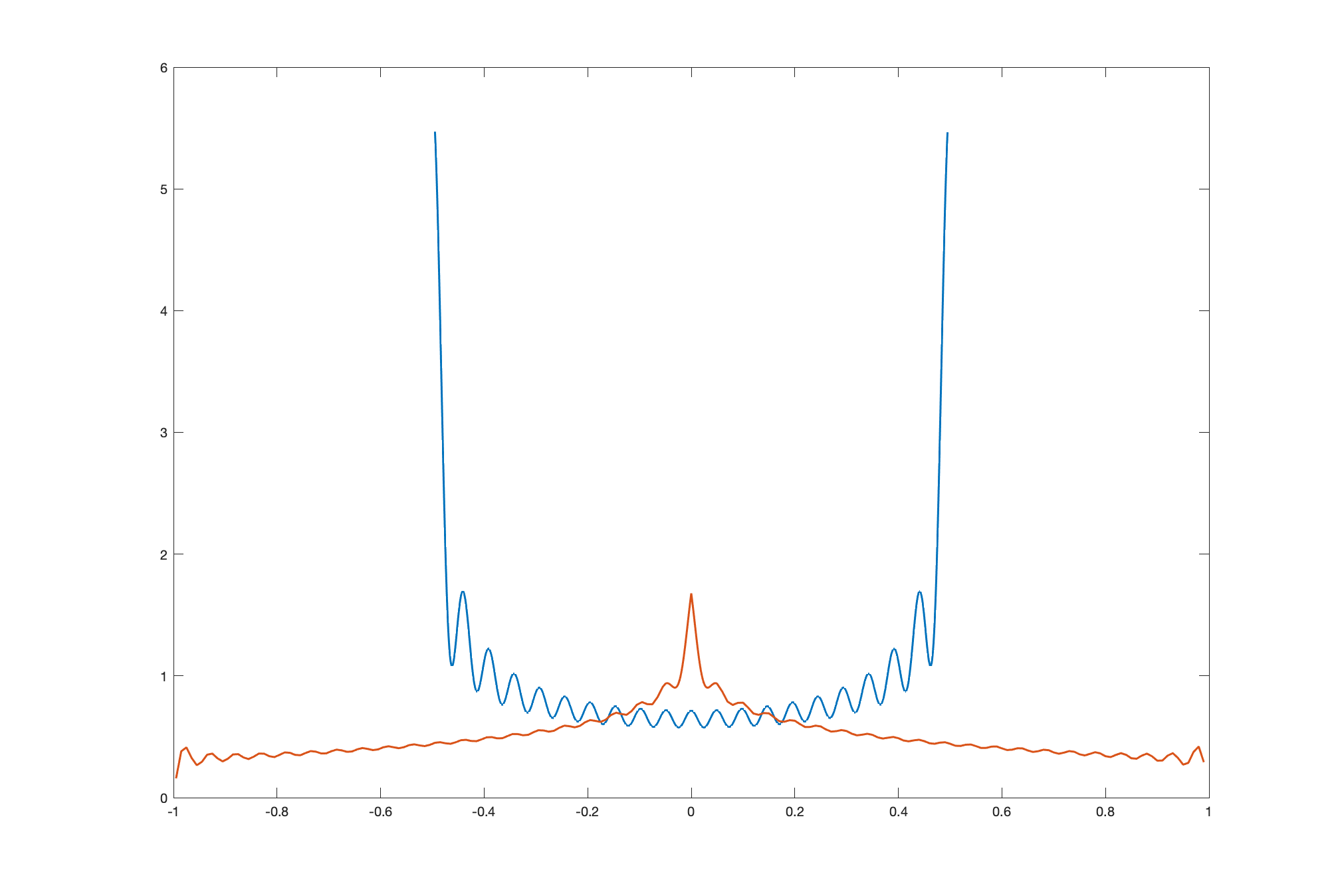}
    \caption{Degree 20 trigonometric polynomial shown in blue and its autoconvolution shown in orange. }
  \label{plot}
\end{figure}

For the remainder of this section, let $T = 30000$, $R = 4000$ and $f_P(x) = 1+\sum_{|k| \leq T} f_{|k|} e^{2\pi i kx}$, where $f_k$ is the solution partially stated above. Also let $F_P$ be the extension of $f_P$ to $[-1,1]$, defined to be zero outside of $[-1/2,1/2]$. In the following two subsections we calculate upper and lower bounds for $\mu_2^2$ thereby proving Theorem~\ref{MT}. We export our computed solution $\{f_k\}_{k=1}^{30000}$ to Matlab and use Variable-Precision Arithmetic' operations to avoid floating-point rounding errors on the order of precision stated in our theorem, we used the default of 32 significant digits. In the calculation of our upper and lower bounds, we will use the following quantities related to $\{f_k\}_{k=1}^{30000}$. 
\begin{equation}\label{fpstat} \sum_{k=1}^T |f_k| \approx 118.753760895, \quad \sum_{k=1}^T |f_k|^3 \approx 0.074874022, \quad  \sum_{k=1}^T k^2|f_k|^3 <157091.\end{equation}

\subsection{Computing an upper bound}

We want to estimate $ \| f_P \ast f_P\|_2^2$ from above. We will take advantage of the fact that the Fourier coefficients $\hat{F}_P(k)$ will decay quickly. From \eqref{f2Ffour} we see that $\hat{F}_P(2m) = 0$ for all $|m| \geq T+1$. Also, for odd $|m| \geq 4T$ we have 
\begin{equation}\label{Fdecay} |\hat{F}_P(m)| = \bigg|\sum_{k\in \mathbb{Z}} \frac{\hat{f}_P(k)(-1)^k}{\pi  (2k-m)} \bigg| \leq \frac{2}{m \pi} \sum_{|k| \leq T} |\hat{f}_P(k)|.\end{equation}
From \eqref{fpstat} we obtain the estimate $|\hat{F}_P(m)| < 152/m$. This gives the bound on the tail sum for all $N \in \mathbb{N}$:
\[ \sum_{m > N} |\hat{F}(2m-1)|^4 < 152^4\int_N^\infty (2x-1)^{-4} = 152^4(2N-1)^{-3}/6.  \]
From \eqref{Meq1} we have for all $N \geq 2T$:
\begin{align*}  \| f_P \ast f_P\|_2^2 &= 8 \sum_{m \in \mathbb{Z}} |\hat{F}_P(m)|^4  \leq  \frac{1}{2} + \sum_{m=1}^T \hat{f}_P(m)^4\\
& +\frac{16}{\pi^4}\sum_{m=1}^N \left( \frac{1}{2m-1} +2(2m-1) \sum_{k=1}^T \frac{\hat{f}_P(k)(-1)^k}{(2m-1)^2-4k^2} \right)^4 + 152^4(2N-1)^{-3}/3 .\end{align*}
The choice of $N = 10^6$ gives the estimate $\| f_P \ast f_P\|_2^2 \leq 0.574643711$.

\subsection{Computing a lower bound}

We use Lemma~\ref{holdlem} to compute a good lower bound. To do this, we need to find a good choice of $g(x)$ on $[-1/2,1/2]$. Equation~\eqref{f2g} suggests that a good choice $g_P$ should have the Fourier coefficients $\hat{g}_P(0) = 2$ and
\[ \hat{g}_P(m) = \frac{2}{1-2\alpha} \hat{f}_P(m)^3, \quad m \in \mathbb{Z}\setminus\{0\}.\]
We can optimize $\alpha$ to suit our particular $f_P$, this ends up giving $\alpha = 0.574975$, which is reasonably close to the true value of $\mu_2^2$. Let $G_P$ be as in the statement for Lemma~\ref{holdlem}, i.e. 
\[ G_P(x) = \begin{cases} 1 & \text{for } x \in [-1/2,1/2] \\ 1-g_P(x) & \text{otherwise.} \end{cases}\]
We need to accurately bound $\sum_{m \neq 0} |\hat{G}_P(m) |^{4/3}$ from above. We can proceed similar to our recent upper bound calculation, using the decay of the Fourier coefficients. For $m \neq 0$ we have 
\[ \hat{G}_P(m) = -\frac{1}{2}(-1)^m \int_{-1/2}^{1/2} g_P(x) e^{-\pi i m x} \ dx.\]
We calculated the explicit dependance of $\hat{G}_P$ on $\hat{g}_P$ in \eqref{f2Ffour}, we restate it below. 
\begin{equation*} \hat{G}_P(m) = \begin{cases} -\frac{1}{2} \hat{g}_P(m/2) & \text{if } m \text{ is even} \\ (-1)^{(m+1)/2}\sum_{k\in \mathbb{Z}} \frac{\hat{g}_P(k)(-1)^k}{\pi  (2k-m)} & \text{if } m \text{ is odd.} \end{cases} \end{equation*}
We see $\hat{G}_P(2m) = 0$ for all $|m| \geq T+1$. Fix an odd $m \in \mathbb{Z}$, similar to \eqref{Fdecay} we have 
\[ |\hat{G}_P(m)| = \bigg|\sum_{k\in \mathbb{Z}} \frac{\hat{g}_P(k)(-1)^k}{\pi  (2k-m)} \bigg| =\frac{2}{\pi m} \bigg|1+\sum_{k=1}^T \frac{\hat{g}_P(k)(-1)^k}{1-4k^2/m^2} \bigg|. \]
Define $1+\theta_k = 1/(1-4k^2/m^2)$. Then if $|m| \geq 5T$ we have $|\theta_k| \leq 5k^2/m^2$. It follows that 
\[ \bigg|1+\sum_{k=1}^T \frac{\hat{g}_P(k)(-1)^k}{1-4k^2/m^2} \bigg| \leq  \bigg|1+\sum_{k=1}^T \hat{g}_P(k)(-1)^k\bigg| + \sum_{k=1}^T \theta_k|\hat{g}_P(k)|. \]

We have the following estimate using \eqref{fpstat}.

\[ |\hat{G}_P(m)| < \frac{2}{\pi  m} \left(\bigg| 1+\frac{2(0.074\ldots)}{1-2(0.574\ldots)}\bigg|  + 1.05\cdot 10^7/m^2 \right) < 0.00087/m,\]

for all odd $|m| \geq 30T$. This gives the bound on the tail sum for all $N \geq  30T$:
\[ \sum_{m > N} |\hat{G}(2m-1)|^{4/3} < 0.00087^{4/3}\int_N^\infty (2x-1)^{-4/3} < 0.000125(2N-1)^{-1/3}.  \]
We compute that 
\[ \sum_{0 \neq |m| \leq 10^6} |\hat{G}_P(m) |^{4/3}  = 1.885128081\ldots. \]
Hence $\sum_{ |m| \neq 0} |\hat{G}_P(m) |^{4/3} < 1.885129002\ldots$. By Lemma~\ref{holdlem} we conclude 
\[ \mu_2^2 \geq \frac{1}{2} + \frac{1}{2 \cdot (1.885129002\ldots)^3} > 0.574635728.\]
This concludes the proof of Theorem~\ref{MT}.

\section{Number theoretic corollaries} 

In this section we briefly discuss how results on $B_h[g]$ sets can be obtained from our estimates of $\mu_2$. We rely heavily on the method of Green~\cite{BG}. The cornerstone of several of the number theoretic results proved by Green is the following. 

\begin{theorem}[Green \cite{BG}, Theorem 6.]\label{BGthm} Let $H \colon \{1,\ldots,N\} \to \mathbb{R}$ be a function such that $\sum_{j=1}^n H(j) = N$, and $v,X$ be positive integers. For each $r \in Z_{2N+v}$ put
\[ \hat{H}(r) = \sum_{x \in Z_{2N+v}} e^{2\pi i rx/(2N+v)} H(x),\]
the discrete Fourier transform. Let $g \in C^1[-1/2,1/2]$ be such that $\int_{-1/2}^{1/2} g(x) \ dx= 2$. Then there is a constant $C$, depending only on $g$ such that 
\[ \sum_{0<|r|\leq X} |\hat{H}(r)|^4 \geq \gamma(g)N^4\bigg( 1 - C\bigg( \frac{v}{N}+\frac{N^2}{v^2X}+\frac{X^2}{N}\bigg)\bigg),\]
where 
\[ \gamma(g) = 2\left( \sum_{r \geq 1 } |\tilde{g}(r/2)|^{4/3}|\right)^{-3}.\]

\end{theorem}

Green finds a function $g \in C^1[-1/2,1/2]$ such that $\gamma(g)>1/7$. Note that Lemma~\ref{holdlem} shows the existence of a $g \in L^2[-1/2,1/2]$ such that $\gamma(g) = 2\mu_2^2-1$. Following a similar argument to Lemma~\ref{L2lem} we can smooth out $g$ into a $C^1$ function while increasing $\gamma(g)$ an arbitrarily small amount. We conclude that Theorem~\ref{BGthm} can be stated by replacing $\gamma(g)$ with $2\mu_2^2-1$. \\

%\noindent \textbf{Proof of bounds in Corollary~\ref{NTcor}.} 

\begin{proof}[Proof of Corollary \ref{NTcor}] To obtain the three bounds on $\sigma_h[g]$ we simply run through the method of Green, replacing the $1/7$ bound with $2\mu_2^2-1$. The bound on $\sigma_4(1)$ is found on \cite[Equation 30]{BG} and stated in \cite[Theorem 15]{BG}. The bound for $\sigma_3(1)$ is obtained through \cite[Lemma 16]{BG} and stated in \cite[Theorem 17]{BG}. Finally, the bound for $\sigma_2(g)$ is obtained by replacing $8/7$ with $2\mu_2^2$ in \cite[Equation 37]{BG}, thereby giving an improved version of \cite[Theorem 24]{BG}. \end{proof}

Lastly, in our proof of Corollary \ref{AEcor} we scale a function on $[N]$ to a simple function on $[0,1]$, and check that the inequalities work in our favour.

\begin{proof}[Proof of Corollary \ref{AEcor}] Let $H \colon [N] \to \mathbb{R}$ be a function with $\sum_{j=1}^N H(j) = N$. Recall the definition of discrete convolution 
\[ H \ast H(x) = \sum_{j=1}^N H(j)H(x-j),\]
and so the additive energy of $H$ is given by $\sum_{j=1}^N H \ast H(x)^2$.
Define the simple function $f \colon [0,1) \to \mathbb{R}$ by 
\[ f(x) = \sum_{j=1}^N H(j) \ind_{((j-1)/N,j/N]}(x).\]
Clearly $f(x-1/2) \in \mathcal{F}$, and so $\|f \ast f\|_2^2 \geq \mu_2^2$. The function $f\ast f$ consists of $2N$ line segments with domain $((j-1)/N,j/N)$ for $j \in [2N]$. Moreover, for all $j \in [2N]$ we have 
\begin{equation}\label{fHrel} Nf \ast f((j-1)/N) = H \ast H (j).\end{equation}
For any line segment $\ell \colon [a,b] \to \mathbb{R}$ we have the following estimate by convexity. 
\begin{equation}\label{lscon} \int_a^b \ell(x)^2 \ dx \leq \frac{\ell(a)^2+\ell(b)^2}{2}(b-a).\end{equation}
By \eqref{lscon} we have
\[ \int_0^2 f\ast f(x)^2 \ dx = \sum_{j=1}^{2N} \int_{(j-1)/N}^{j/N} f\ast f (x)^2 \ dx \leq \frac{N}{2}\sum_{j=1}^{2N}\big(f \ast f( j/N)^2 +f \ast f( (j-1)/N)^2 \big).  \]
And by \eqref{fHrel} the above becomes
\[ \frac{N^3}{2}\sum_{j=1}^{2N}\big(H \ast H( j)^2 +H \ast H( j-1)^2 \big) = N^3 \sum_{j=1}^{2N} H \ast H(j)^2 .\]
This proves the additive energy bound.

\end{proof}

\vspace{1cm}
\textbf{Acknowledgements.} The author thanks Greg Martin, Kevin O'Bryant, and Jozsef Solymosi for helpful comments during the preparation of this work.

\vspace{1cm}

\begin{small}
\noindent {\sc Department of Mathematics  \\ The University of British Columbia \\Room 121, 1984 Mathematics Road \\ Vancouver, BC \\ Canada V6T 1Z2} \\
\url{epwhite@math.ubc.ca}\\

\end{small}

\clearpage

\section{Appendix: A simple almost-optimizer}

We discuss functions of the form 
\[ f_c(x) = \frac{\alpha_c}{(1/4-x^2)^c}, \quad \alpha_c = \frac{\Gamma(2-2c)}{\Gamma^2(1-c)}.\]
Note that $f_c \in \mathcal{F}$ for all $c <1$, by straightforward calculation of a Beta function. We thank Kevin O'Bryant for suggesting this family. The function $f_{1/2}$ is the probability density function of the arcsine distribution, and has appeared in other works on autoconvolutions such as \cite{BS,MO2}. The reason for its relevance to autoconvolutions is perhaps the fact that the singularities at $x = \pm 1/2$ are precisely the right magnitude so that $f \ast f$ neither goes to 0, nor blows up at $x = \pm 1$. This makes $f \ast f$ flatter overall compared to a function that vanishes or blows up at $x= \pm 1$, and so the norms of $f \ast f$ may be smaller than most other functions. 

\vspace{0.1cm}

To numerically evaluate $\|f_c \ast f_c\|_2$ we use \eqref{Meq1}. To calculate the Fourier coefficients we apply the Mehler-Sonine formula \cite[\S 3.71 (9)]{GW}:

\[ J_\nu(z) = \frac{\left(\frac{z}{2} \right)^\nu}{\Gamma(\nu + \frac{1}{2}) \sqrt{\pi} }\int_{-1}^1 e^{i z s} (1-s^2)^{\nu - 1/2} \ ds, \quad z \in \mathbb{C}, \nu > -1/2,\]
where $J_\nu(z)$ is the order $\nu$ Bessel function of the first kind. This gives the Fourier coefficients for $k \neq 0$:
\[ \hat{F}_c(k) = \frac{1}{2} \int_{-1/2}^{1/2} e^{\pi i k x} f_c(x) =  \frac{\pi^c \Gamma(2-2c)}{2 \Gamma(1-c)} \cdot \frac{J_{1/2-c}(\pi k/2)}{k^{1/2-c}},\]
where $F_c$ is the extension of $f_c$ onto $[-1,1]$, zero outside of $[-1/2,1/2]$. It follows from \eqref{Meq1} that  
\[ \| f_c \ast f_c \|_2^2 = 8 \sum_{k \in \mathbb{Z}} |\hat{F}_c(k)|^4 = \frac{1}{2} +  \frac{\pi^{4c}\ \Gamma^4(2-2c)}{\Gamma^4(1-c)}\sum_{k=1}^\infty \frac{J_{1/2-c}^4(\pi k/2)}{k^{2-4c}}. \]

We estimate the tail of the above series using the bound $|J_{1/2-c}(\pi k) | \leq k^{-1/2}$ for all $k \geq 2$ and $c>0$ \cite[\S 7.1]{GW}. Using Matlab, we numerically determined the optimal $f_c$ to be at $c = 0.4942$ with 
\[ \| f_{0.4942} \ast f_{0.4942} \|_2^2 = 0.5746482,\]
within an error of $5 \cdot 10^{-9}$. We remark that this is only $4.5 \cdot 10^{-6}$ worse than the upper bound found using the convex program.

\end{document}